\documentclass[12pt]{amsart}
\usepackage{amsmath,amssymb,amsbsy,amsfonts,amsthm,latexsym,amsopn,amstext,
                            amsxtra,euscript,amscd}

\newfont{\teneufm}{eufm10}
\newfont{\seveneufm}{eufm7}
\newfont{\fiveeufm}{eufm5}
\newfam\eufmfam
                 \textfont\eufmfam=\teneufm \scriptfont\eufmfam=\seveneufm
                 \scriptscriptfont\eufmfam=\fiveeufm
\def\bbbc{{\mathchoice {\setbox0=\hbox{$\displaystyle\rm C$}\hbox{\hbox
to0pt{\kern0.4\wd0\vrule height0.9\ht0\hss}\box0}}
{\setbox0=\hbox{$\textstyle\rm C$}\hbox{\hbox
to0pt{\kern0.4\wd0\vrule height0.9\ht0\hss}\box0}}
{\setbox0=\hbox{$\scriptstyle\rm C$}\hbox{\hbox
to0pt{\kern0.4\wd0\vrule height0.9\ht0\hss}\box0}}
{\setbox0=\hbox{$\scriptscriptstyle\rm C$}\hbox{\hbox
to0pt{\kern0.4\wd0\vrule height0.9\ht0\hss}\box0}}}}
\def\bbbq{{\mathchoice {\setbox0=\hbox{$\displaystyle\rm
Q$}\hbox{\raise
0.15\ht0\hbox to0pt{\kern0.4\wd0\vrule height0.8\ht0\hss}\box0}}
{\setbox0=\hbox{$\textstyle\rm Q$}\hbox{\raise
0.15\ht0\hbox to0pt{\kern0.4\wd0\vrule height0.8\ht0\hss}\box0}}
{\setbox0=\hbox{$\scriptstyle\rm Q$}\hbox{\raise
0.15\ht0\hbox to0pt{\kern0.4\wd0\vrule height0.7\ht0\hss}\box0}}
{\setbox0=\hbox{$\scriptscriptstyle\rm Q$}\hbox{\raise
0.15\ht0\hbox to0pt{\kern0.4\wd0\vrule height0.7\ht0\hss}\box0}}}}
\def\bbbt{{\mathchoice {\setbox0=\hbox{$\displaystyle\rm
T$}\hbox{\hbox to0pt{\kern0.3\wd0\vrule height0.9\ht0\hss}\box0}}
{\setbox0=\hbox{$\textstyle\rm T$}\hbox{\hbox
to0pt{\kern0.3\wd0\vrule height0.9\ht0\hss}\box0}}
{\setbox0=\hbox{$\scriptstyle\rm T$}\hbox{\hbox
to0pt{\kern0.3\wd0\vrule height0.9\ht0\hss}\box0}}
{\setbox0=\hbox{$\scriptscriptstyle\rm T$}\hbox{\hbox
to0pt{\kern0.3\wd0\vrule height0.9\ht0\hss}\box0}}}}
\def\bbbs{{\mathchoice
{\setbox0=\hbox{$\displaystyle     \rm S$}\hbox{\raise0.5\ht0\hbox
to0pt{\kern0.35\wd0\vrule height0.45\ht0\hss}\hbox
to0pt{\kern0.55\wd0\vrule height0.5\ht0\hss}\box0}}
{\setbox0=\hbox{$\textstyle        \rm S$}\hbox{\raise0.5\ht0\hbox
to0pt{\kern0.35\wd0\vrule height0.45\ht0\hss}\hbox
to0pt{\kern0.55\wd0\vrule height0.5\ht0\hss}\box0}}
{\setbox0=\hbox{$\scriptstyle      \rm S$}\hbox{\raise0.5\ht0\hbox
to0pt{\kern0.35\wd0\vrule height0.45\ht0\hss}\raise0.05\ht0\hbox
to0pt{\kern0.5\wd0\vrule height0.45\ht0\hss}\box0}}
{\setbox0=\hbox{$\scriptscriptstyle\rm S$}\hbox{\raise0.5\ht0\hbox
to0pt{\kern0.4\wd0\vrule height0.45\ht0\hss}\raise0.05\ht0\hbox
to0pt{\kern0.55\wd0\vrule height0.45\ht0\hss}\box0}}}}
\def\bbbz{{\mathchoice {\hbox{$\sf\textstyle Z\kern-0.4em Z$}}
{\hbox{$\sf\textstyle Z\kern-0.4em Z$}}
{\hbox{$\sf\scriptstyle Z\kern-0.3em Z$}}
{\hbox{$\sf\scriptscriptstyle Z\kern-0.2em Z$}}}}

 \newtheorem{thm}{Theorem}
 
 \newtheorem{lem}[thm]{Lemma}
 
 \theoremstyle{definition}
 
 \theoremstyle{remark}

\def\cX{{\mathcal X}}

\def\({\left(}
\def\){\right)}
\def\[{\left[}
\def\]{\right]}
\def\<{\langle}
\def\>{\rangle}

\def\fB{{\mathfrak B}}

\def\F{\mathbb{F}}

\def\Z{\mathbb{Z}}

\def\ep{{\mathbf{\,e}}}

\def\vec#1{\mathbf{#1}}

\def\mand{\qquad\mbox{and}\qquad}

\begin{document}

\title[Points on  Markoff-Hurwitz and Dwork Hypersurfaces]{On the Density
of Integer 
Points on the Generalised Markoff-Hurwitz and Dwork Hypersurfaces}

 \author[I. E. Shparlinski] {Igor E. Shparlinski}

\address{Department of Pure Mathematics, University of New South Wales,
Sydney, NSW 2052, Australia}
\email{igor.shparlinski@unsw.edu.au}

\begin{abstract}  
We use bounds of mixed  character sums modulo a prime $p$
to estimate the density of integer points on the hypersurface 
$$
f_1(x_1) + \ldots + f_n(x_n) =a x_1^{k_1} \ldots x_n^{k_n} 
$$
for some polynomials $f_i \in \Z[X]$, nonzero integer $a$
and positive integers $k_i$ 
$i=1, \ldots, n$. 
In the case of 
$$
f_1(X) = \ldots = f_n(X) = X^2\mand
k_1 = \ldots = k_n =1
$$ 
the above congruence is known as the Markoff-Hurwitz hypersurface,
while for 
$$
f_1(X) = \ldots = f_n(X) = X^n\mand
k_1 = \ldots = k_n =1
$$ 
it is known as the Dwork hypersurface. Our result is 
substantially stronger than those known for 
general supersurfaces.
\end{abstract}



\maketitle

\section{Introduction}

Studying the density of integer and rational 
points $(x_1, \ldots, x_n)$ on hypersurfaces 
has always been an active 
area of research, where many rather involved methods
have led to remarkable achievements, see~\cite{Brow,BrHBSa,H-B,H-BP1,Mar1,Mar2,Salb1,Salb2,Tsch} 
and references therein. More precisely, given a  hypersurface 
$$
F(x_1, \ldots, x_n) = 0
$$
defined by a polynomial $F \in \Z[X_1, \ldots, X_n]$ in $n$ variables,
the goal is to estimate the number $N_F(\fB)$ of solutions 
$(x_1, \ldots, x_n) \in \Z^n$ that fall in a hypercube $\fB$ of the 
form 
\begin{equation}
\label{eq:box}
 \fB = [u_1+1,u_1+h]\times \ldots \times [u_n+1, u_n+h]. 
\end{equation}
Unfortunately, even in the most favourable situation, 
the currently known general approaches lead only to a  bound of 
the form $N_F(\fB) = O\(h^{n-2+\varepsilon}\)$ for any fixed 
$\varepsilon>0$ or even weaker, see~\cite{BrHBSa,H-BP1,Salb1,Salb2}.
For some special types of hypersurfaces the strongest known bounds are due 
 Heath-Brown~\cite{H-B} and 
Marmon~\cite{Mar1,Mar2}.
For example, for hypercubes around  the origin, Marmon~\cite{Mar2}
gives a bound of the form  $N_F(\fB) = O\(h^{n-4+\delta_n}\)$ for 
a class of hypersurfaces, 
with some explicit function $\delta_n$ such that $\delta_n  \sim 37/n$ as 
$n \to \infty$.
Combining this bound with some previous results and 
methods, for a certain class of hypersurfaces,  
Marmon~\cite{Mar2} also derives the bound 
$N_F(\fB) = O\(h^{n-4+\delta_n}+ h^{n-3+\varepsilon}\)$
which holds for an arbitrary hypercube $\fB$ with any fixed $\varepsilon>0$ 
and the implied constant that depends only of $\deg F$, $n$ and $\varepsilon$
(note that  $\delta_n> 1$ for $n < 29$). 
We also remark that when the number of variables $n$ is exponentially large 
compared to $d$ and the highest degree form of $F$ is non-singular, then
the methods developed as the continuation 
of the work of Birch~\cite{Birch} lead to much stronger 
bounds, of essentially optimal order of magnitude.

Here, we show that in some interesting special cases, to which 
further developments of~\cite{Birch} do not apply
(as the highest degree form is singular and the number of variables
is not large enough) 
a modular approach  leads to a bound of the form 
$$
N_F(\fB) = O\(h^{n-4+15/(n-6)^{1/2}}\)
$$
for a sufficiently large $n$. We note that although this bound is of a similar shape 
as that of Mormon~\cite[Theorem~1.1]{Mar2} these results apply to very different 
classes of hypersurfaces (and as we have mentioned, the result of~\cite{Mar2}
applies only to hypercubes $\fB$ at the origin).

More precisely we concentrate on hypersurfaces of the form
\begin{equation}
\label{eq:MHD eq}
f_1(x_1) + \ldots + f_n(x_n)  = ax_1^ {k_1} \ldots x_n^{k_n} 
\end{equation}
defined by some polynomials $f_i \in \Z[X]$, a non-zero integer $a$
and positive integers $k_i$, $i=1, \ldots, n$.
In particula, we use $N_{a, \vec{f}, \vec{k}}(\fB)$
to denote the number of integer solutions to~\eqref{eq:MHD eq}
with $(x_1, \ldots, x_n) \in \fB$, where $\vec{f} = (f_1,\ldots, f_n)$
and $\vec{k} = (k_1,\ldots, k_n)$.

In the case of 
$$
f_1(X) = \ldots = f_n(X) = X^2\mand
k_1 = \ldots = k_n =1
$$ 
the equation~\eqref{eq:MHD eq}  defines the {\it Markoff-Hurwitz hypersurface\/},
see~\cite{Bar1,Bar2,Bar3,Cao}, where various questions related to
these   hypersurfaces have been investigated.

Furthermore,   for 
$$
f_1(X) = \ldots = f_n(X) = X^n\mand
k_1 = \ldots = k_n =1
$$ 
the equation~\eqref{eq:MHD eq} is known as the {\it Dwork hypersurface\/}, which has been  intensively 
studied by various authors~\cite{Gou,HS-BT,Katz,Kloost,Yu}, in particular, 
as an example of a {\it Calabi-Yau variety\/}.

Here, we use some ideas from~\cite{Shp}  combined 
with some the results of~\cite{BGKS} to show that 
if  
$$
\max \deg f_i \le D, \qquad i=1, \ldots, n, 
$$
and $k_1,\ldots, k_n\ge $ are   odd, 
then, for an  arbitrary  $\varepsilon> 0$  and $n \ge n_0(\varepsilon, D)$, 
where $n_0(\varepsilon, D)$ depends only on $\varepsilon$ and $D$, for any
hypercube $\fB$ of the form~\eqref{eq:box} we have
$$
N_{a, \vec{f}, \vec{k}}(\fB) = O\(h^{n-4+\varepsilon}\)
$$
uniformly over $u_1, \ldots, u_n$. 

%


Throughout the paper, the implied constants in the symbols ``$O$'', ``$\ll$'' and
``$\gg$'' may depend on some positive real parameters $\varepsilon$  and $\delta$, 
the  polynomials $\deg f_i$  and the exponents $k_{i}$ in~\eqref{eq:MHD eq},
$i=1, \ldots, n$.  We recall that
the expressions $A=O(B)$, $A \ll B$ and $B \gg A$ are each equivalent to the
statement that $|A|\le cB$ for some constant $c$.

\section{Character and Exponential Sums} 

Here we fix some sufficiently large prime $p$ and
let $\cX$ be the set of multiplicative characters modulo $p$
and let $\cX^* = \cX \setminus \{\chi_0\}$ be the 
set of non-principal characters (we set $\chi(0) = 0$ for all $\chi \in \cX$).

We also denote
$$
\ep(z) = \exp(2 \pi i z/p).
$$
We appeal to~\cite{IwKow} for a background 
on the basic properties of multiplicative characters and
exponential functions, such as orthogonality. 

First we need the following well-know property of Gauss sums
$$
G(\chi, \lambda) = \sum_{y=1}^{p-1} \chi(y)\ep(\lambda y), \qquad 
\chi \in \cX, \ \lambda \in \F_p,
$$
see~\cite[Section~3.4]{IwKow}.

\begin{lem}
   \label{lem:Gauss}
For  any $\chi \in \cX$ and $\lambda \in \F_p$, we have
$$
|G(\chi, \lambda)| = 
\left\{\begin{array}{ll}
1, & \text{for $\chi=\chi_0$, $\lambda \ne 0$},\\
0, & \text{for $\chi\ne \chi_0$, $\lambda = 0$},\\
p^{1/2}, & \text{for $\chi\ne \chi_0$, $\lambda \ne 0$}. 
\end{array}
\right. 
$$
\end{lem}

We also need a  bound of exponential sums twisted with a
multiplicative character has been given by D. R. Heath-Brown and  
Pierce~\cite{H-BP2} is an improvement of a recent resuly of Chang~\cite{Chang}.
We present a result of~\cite{H-BP2} in a somewhat simplified form,
which is sufficient for our applications. 

\begin{lem}
   \label{lem:MixedSum}
There is a function $\kappa(z)$ with 
$$
\lim_{z\to 0} \kappa(z)/z^2 = 1
$$
such that for  any   $\chi \in \cX^*$,  polynomial $F(X) \in \F_p[X]$
of degree $s$ 
and  integers $u$ and $h$ with $p^{1/2} \ge h \ge p^{1/4+\delta}$, we have
$$\sum_{x=u+1}^{u+h} \chi(x) \ep(F(x)) \ll 
h p^{-\kappa(\delta)}.
$$
\end{lem}

We note that we do not impose any conditions on the polynomial 
$F$ in Lemma~\ref{lem:MixedSum}, which, in particular can be a constant 
polynomials (in which case, we also have the Burgess bound, 
of course, see~\cite[Theorem~12.6]{IwKow}).
%


%


\section{Congruences with Products}
\label{sec:cong}

For a prime $p$ and integers $h\ge3$, $\nu
\ge 1$ and $k$, we
denote by $I_{p,\nu}(u,h)$ the number of solutions of the
congruence
\begin{equation*}
\begin{split}
(x_1+u)\ldots(x_{\nu}+u)\equiv & (y_1+u)\ldots(y_{\nu}+u)\not\equiv0 \pmod p,\\
1\le x_j,y_j & \le   h_j,  \qquad j =1, \ldots, \nu.
\end{split}
\end{equation*}

As usual, we use $\pi(T)$ to denote the number of primes $p \le T$.

We need the following estimate  from~\cite{BGKS}:

\begin{lem}
\label{lem:GProdSet AlmostAll} Let $\nu \ge 1$ be a fixed integer.
Then for a sufficiently large positive integer  $T \ge3$,
for all but $o(\pi(T))$ primes $p \le T$  
and any integers $u$ and $h<p$, we have the bound
$$
I_{p,\nu}(u,h) \le \(h^\nu + h^{2\nu-1/2}p^{-1/2}\)h^{o(1)}. 
$$
\end{lem}

\section{Main Result}

We are now able to present our main result.
 
\begin{thm}
\label{thm:Bound}  
Let at least two of the polynomials 
$f_1(X),\ldots, f_1(X)\in \Z[X]$ be of 
positive degree and let $k_1,\ldots, k_n\ge 1$ be odd 
integers. 
For a sufficiently small  $\varepsilon > 0$
we have  
$$
 N_{a, \vec{f}, \vec{k}}(p,\fB) \ll h^{n-4 + \varepsilon} 
$$
provided that 
$$
n \ge  225 \varepsilon^{-2} + 6. 
$$
\end{thm}

\begin{proof} Let $p$ be an arbitrary prime.
Clearly 
\begin{equation}
\label{eq:N Np}
N_{a, \vec{f}, \vec{k}}(\fB) \le N_{a, \vec{f}, \vec{k}}(p,\fB), 
\end{equation}
where $N_{a, \vec{f}, \vec{k}}(p, \fB)$
is the number of solutions to the congruence
\begin{equation}
\label{eq:cong}
f_1(x_1) + \ldots + f_n(x_n)  \equiv a x_1^{k_1} \ldots x_n^{k_n} \pmod p
\end{equation}
with $(x_1, \ldots, x_n) \in \fB$.

It is clear that one can choose $p$ in the interval
\begin{equation}
\label{eq:size}
h^{4 - \varepsilon} \le p  \le 2h^{4 - \varepsilon}
\end{equation}
and also in the arithmetic progression 
\begin{equation}
\label{eq:progr}
p \equiv 3 \pmod {k_1\ldots k_n}
\end{equation}
for which the bound of Lemma~\ref{lem:GProdSet AlmostAll} 
holds with $\nu = 3$. 

Note that since $k_1, \ldots, k_n$ are odd, 
the congruence~\eqref{eq:progr} implies that  
\begin{equation}
\label{eq:gcd}
\gcd(k_1\ldots k_n,p-1)=1.
\end{equation}

Without loss of generality, we assume that the polynomials 
$f_1$ and $f_2$ are of 
positive degree.
Further, we can also assume that $p$ is sufficiently large  
so that $\gcd(a,p) =1$ and also the leading coefficients 
of the polynomials $f_1$ and $f_2$ are relatively prime to $p$,
so the positivity of the degree is preserved.

We now proceed as in the proof of~\cite[Theorem~3.2]{Shp}.
Let 
$$
S_i(\chi; \lambda) = \sum_{x=u_i+1}^{u_i+h} \chi^{k_i}(x)
\ep\(\lambda f_i(x) \), \quad i =1, \ldots, n.
$$
Then  by~\cite[Equation~(3.3)]{Shp}, under the 
condition~\eqref{eq:gcd}, we have:
\begin{equation}
\label{eq:R1R2}
 N_{a, \vec{f}, \vec{k}}(p,\fB)
- \frac{h^n}{p} \ll  \frac{1}{p^2}\(R_1  + R_2\), 
\end{equation}
where
\begin{equation*}
\begin{split}
R_1 &=   \sum_{\lambda \in \F_p} \sum_{\chi\in \cX^*}
 |\overline{G}(\chi, \lambda) |
\prod_{i=1}^n |S_i(\chi, \lambda)| ,\\
 R_2 &=   \sum_{\lambda \in \F_p^*}  
 |\overline{G}(\chi_0, \lambda) |
\prod_{i=1}^n |S_i(\chi_0, \lambda)|,
\end{split}
\end{equation*}
and $\overline{G}(\chi, \lambda) $ is the complex conjugate of the Gauss sum. 

To estimate $R_1$ we first use  Lemmas~\ref{lem:Gauss} and~\ref{lem:MixedSum} 
and infer that
$$
R_1 \le  h^{n-6} p^{-(n-6)\kappa(\delta)  +1/2} 
 \sum_{\lambda \in \F_p} \sum_{\chi\in \cX^*}
 \prod_{i=1}^6 |S_i(\chi; \lambda)|,
$$
where the function $\kappa(z)$ is as in Lemma~\ref{lem:MixedSum}
and $\delta$ is given 
\begin{equation}
\label{eq:delta}
\delta = \frac{1}{4-\varepsilon} -\frac{1}{4} = \frac{\varepsilon}{4(4-\varepsilon)}
>  \frac{\varepsilon}{16} . 
\end{equation}

Using the H{\"o}lder inequality, we obtain
$$
 \sum_{\chi\in \cX^*}
 \prod_{i=1}^6 |S_i(\chi; \lambda)| 
 \le  \(\prod_{i=1}^6  \sum_{\chi\in \cX^*}  
 |S_i(\chi; \lambda)|^6\)^{1/6} .
$$

Now using the orthogonality of characters we see that 
$$
\sum_{\chi\in \cX^*} |S_i(\chi; \lambda)|^6 \le
\sum_{\chi\in \cX} |S_i(\chi; \lambda)|^6 = (p-1) I_{p, 3}(u_i,h) .
$$
Applying the bound of Lemma~\ref{lem:GProdSet AlmostAll}, which 
is possible due to our choice of $p$, we derive 
$$
R_1 \ll    h^{n-6+o(1)} p^{-(n-6)\kappa(\delta) +5/2}   \(h^3 + h^{11/2}p^{-1/2}\). 
$$
We see that  under the condition~\eqref{eq:size} we have $h^3 < h^{11/2}p^{-1/2}$
(provided $\varepsilon$ is sufficiently small), 
hence the last bound simplifies as 
\begin{equation}
\label{eq:R1 bound}
R_1 \ll    h^{n-1/2+o(1)} p^{-(n-6)\kappa(\delta) +2}  . 
\end{equation}

For $R_2$ we proceed exactly as in the proof of~\cite[Theorem~3.3]{Shp}
and derive
\begin{equation}
\label{eq:R2 bound}
R_2 \ll    h^{n-1} p.
\end{equation}
Indeed,  it follows immediately  
from Lemma~\ref{lem:Gauss} and the trivial bound 
$$
|S_i(\chi_0; \lambda)| \le h, \qquad i = 3, \ldots, n,
$$
that 
$$
R_2 \le    \sum_{\lambda \in \F_p^*}  
\prod_{i=1}^n |S_i(\chi_0, \lambda)| \le 
h^{n-2}\sum_{\lambda \in \F_p} 
|S_1(\chi_0; \lambda)| |S_2(\chi_0; \lambda)|.
$$
Using the Cauchy inequality and the orthogonality 
of exponential functions, because the polynomials 
$f_1$ and $f_2$ are not constant modulo $p$, we obtain 
\begin{equation*}
\begin{split}
 \sum_{\lambda \in \F_p} 
|S_1(\chi_0; \lambda)|&|S_2(\chi_0; \lambda)| \\
&\le  \(   \sum_{\lambda \in \F_p}
 |S_1(\chi_0; \lambda)|^2 \sum_{\lambda \in \F_p}
 |S_2(\chi_0; \lambda)|^2\)^{1/2} \ll  p h, 
\end{split}
\end{equation*}
which implies~\eqref{eq:R2 bound}.

Substituting the bounds~\eqref{eq:R1 bound} and~\eqref{eq:R2 bound} in~\eqref{eq:R1R2}
we obtain
\begin{equation}
\label{eq:Np bound}
N_{a, \vec{f}, \vec{k}}(p,\fB)
=  \frac{h^n}{p}  + O\( h^{n-1/2+o(1)} p^{-(n-6)\kappa(\delta) }  +  h^{n -1} p^{-1}\).
\end{equation}
Recalling~\eqref{eq:size}, we see that
$$
h^{n-1/2} p^{-(n-6)\kappa(\delta) } \ll  h^n p^{-1}
$$
provided that 
\begin{equation}
\label{eq:cond n}
n \ge \kappa(\delta)^{-1} \(1 - \frac{1}{2(4 -\varepsilon)}\) + 6.
\end{equation}
Hence, in this case, combining~\eqref{eq:N Np} and~\eqref{eq:Np bound}, 
we obtain the desired bound.

Clearly, for any $\varepsilon>0$ we have 
$$
1 - \frac{1}{2(4 -\varepsilon)} < \frac{7}{8}.
$$
Furthermore,  we see from the property of the function $\kappa(z)$ 
and~\eqref{eq:delta} that for a sufficiently small $\varepsilon$ 
we see also 
$$
\kappa(\delta) \ge \frac{\varepsilon^2}{257}  .
$$
The result now follows from~\eqref{eq:cond n}. 
\end{proof}

\section{Comments}
\label{sec:Comm}

In Theorem~\ref{thm:Bound} the restriction on $n$ is chosen to 
guarantee the strongest possible bound $O(h^{n-4+\varepsilon})$ 
achieved within our approach.
Certainly for smaller values of $n$, using other estimates 
from~\cite{H-BP2} instead of Lemma~\ref{lem:MixedSum}, 
one can still get bounds stronger than
$O(h^{n-2})$ for smaller values of $n$ (the choice of $p$ has 
also to be modified 
too in order to achieve optimal results). 

Clearly the strength of the bound $O(h^{n-4 +\varepsilon})$
of Theorem~\ref{thm:Bound} is the limit of our method,
unless the range of $h$ in Lemma~\ref{lem:MixedSum}
is expanded. 
However one can possibly hope to reduce the lower bound on the 
number of variables $n$. Furthermore,  besides Lemma~\ref{lem:MixedSum} 
it also depends on the strength of the bound in 
Lemma~\ref{lem:GProdSet AlmostAll}.
Here in some cases one can do better.
We essentially need to show the existence of a prime $p$ in 
a dyadic interval $[T, 2T]$ with a small value of $I_{p,\nu}(u,h)$.
It is easy to see 
that 
\begin{equation*}
\begin{split}
\sum_{p \in [T, 2T]}  I_{p,\nu}(u,h)
\le& \( \pi(2T)-\pi(T) \) K_{\nu}(u,h)
\\
&+ \sum_{v_1,w_1,\ldots, v_{\nu},w_{\nu}=1}^h 
\omega\(|v_1\ldots v_{\nu}- w_1 \ldots w_{\nu}|\),
\end{split}
\end{equation*}
where $K_{\nu}(u,h)$ is the number of integer solutions to the 
equation
\begin{equation*}
\begin{split}
(x_1+u)\ldots&(x_{\nu}+u)= (y_1+u)\ldots(y_{\nu}+u),\\
1\le &x_j,y_j  \le   h_j,  \qquad j =1, \ldots, \nu, 
\end{split}
\end{equation*}
and $\omega(m)$ denotes the number of prime divisors 
of an integer $m$ (where we set $\omega(1) = \omega(0) = 0$). 
If $u$ is not too large compared to $h$ then this approach 
leads to a stronger result (we also refer to~\cite{BGKS} 
for various bounds on $K_{\nu}(u,h)$). In particular,
it is now easy to show that if $u = h^{O(1)}$ then 
$$
\frac{1}{\( \pi(2T)-\pi(T) \)}
\sum_{p \in [T, 2T]}  I_{p,\nu}(u,h)
\le h^{\nu + o(1)} + h^{2\nu + o(1)}T^{-1+o(1)}.
$$

It is also interesting to remove the condition on the parity 
of $k_1, \ldots, k_n$. If some of $k_1, \ldots, k_n$ are even that we 
take $p$ to satisfy 
$$
p \equiv 3 \pmod {2k_1\ldots k_n}
$$
instead of~\eqref{eq:progr}, and then instead of~\eqref{eq:gcd}
we obtain
$$
\gcd(k_i,p-1) \le 2.
$$
We now have to consider separately the contribution from 
the quadratic character $\chi_2$, namely,
$$
R_3 = \sum_{\lambda \in \F_p} 
 |\overline{G}(\chi_2, \lambda) |
\prod_{i=1}^n |S_i(\chi_2, \lambda)|.
$$
Clearly, if  $k_i$ is even that  Lemma~\ref{lem:MixedSum}
does not apply to $|S_i(\chi_2, \lambda)|$.  However, if $f_i$ is a
of degree $\deg f_i \ge 2$, 
one can use instead  estimates of exponential sums with polynomials, 
for example, the bound of Wooley~\cite{Wool3}.
The strength of the final result obtained along these lines, 
depends on the various assumptions 
on the degrees of $f_1, \ldots, f_n$ and on the number of even
integers among $k_1, \ldots, k_n$.



\section{Acknowledgment}


The author would like to thank Roger Heath-Brown and Lillian  Pierce
for informing him about their work~\cite{H-BP2} when it was still 
in progress and then sending him a preliminary draft. 
The author is also grateful tp Oscar Marmon for many useful comments.

This work was supported in part by the  ARC Grant DP130100237.

\end{document}